\documentclass[reqno]{amsart} 
\usepackage[utf8]{inputenc}
\usepackage{amscd}
\usepackage[ngerman, english]{babel}
\usepackage{amsfonts, amsmath,amsthm,booktabs,amssymb}
\usepackage{url}
\usepackage{mathrsfs}
\usepackage{
epsf}
\newtheorem{theorem}{Theorem}[section]
\newtheorem{proposition}[theorem]{Proposition}
\newtheorem{lemma}[theorem]{Lemma}
\newtheorem{corollary}[theorem]{Corollary}

\theoremstyle{remark} \newtheorem{remark}[theorem]{Remark}

\theoremstyle{definition}
\newtheorem{defn}{Definition}[section]

\begin{document}
\title[Determinant bundles on moduli spaces]{Determinant   bundles and geometric quantization of vortex moduli spaces on compact ${\bf K \ddot{A}HLER}$  surfaces }

\author{ Saibal Ganguli*}

\address{*  Bhaskaracharya Pratisthan.,\\
56/14 Erandavane,Damle path\\
Pune 411004, India.\\
email: saibalgan@gmail.com}

\begin{abstract}
In this paper we first show that on projective manifolds $(M, \omega)$, there are holomorphic determinant bundles (in the sense of Knusden-Mumford used by  Bismut, Gillet, Soul$\acute{\rm{e}}$) which play the role of 
the geometric quantum bundle, namely one for each  input data of a Hermitian  holomorphic line bundle $L$ of non-trivial Chern class on  a compact K$\ddot{\rm{a}}$hler manifold $Z$ (with Todd genus non-zero)  and a choice of a geometric quantization of $(M, \omega)$.  Next we further study the generalization of the vortex equations on 
K$\ddot{ \rm{a}}$hler $4$-manifold which has been studied earlier   by Bradlow. We show that when the K$\ddot{ \rm{a}}$hler $4$-manifold  avoids some obstructions then the regular part of the moduli space is a  K$\ddot{ \rm{a}}$hler manifold  and admit a 
  pull back of a Quillen determinant  bundle as the  quantum line bundle, i.e. the curvature is proportional to the K$\ddot{ \rm{a}}$hler  form. Thus they can be
  quantized geometrically. In fact we show that  the moduli space of the usual vortex equations on a projective K$\ddot{\rm{a}}$hler $4$-manifold is 
projective when the moduli space is smooth. Since in
K$\ddot{\rm{a}}$hler $4$-manifold the vortex moduli and the Seiberg Witten moduli coincide our effort gives a quantization of Seiberg Witten moduli by determinant bundles.
\end{abstract}

\maketitle

\vskip 5mm

Mathematics  Subject Classification (MSC2010) :  mathematical physics ( 81V99), differential geometry (58A32).

\section{{\bf Introduction}}

Given a symplectic manifold $(M, \omega)$ with $\omega $ integral (i.e.  its cohomology class is in the image of $H^2(M, {\mathbb Z}) \rightarrow H^2(M, {\mathbb R})$), 
geometric prequantization is the construction of a Hermitian line bundle with a connection (called the prequantum line bundle) whose curvature $\rho$ is proportional to the 
symplectic form $\omega$. This is always possible as long as $\omega$ is integral. This method of quantization, developed by Kostant and Souriau,  assigns to 
functions $f \in C^{\infty}(M)$, an operator, $\hat{f} =-i \nabla_{X_f} +   f$ acting on the Hilbert space of square 
integrable sections of $L$ (the wave functions). Here $\nabla = d - i \theta$ where 
locally $\omega = d \theta$ and $X_f$ is defined by $\omega(X_f, \cdot) = - df(\cdot)$. We have taken $ \hbar=1$.  The general reference for this is Woodhouse ~\cite{W}.  

This assignment has the property that  that the Poisson bracket (induced by  the symplectic form), namely, 
$\{ f_1, f_2 \}_{PB} $ corresponds  to an operator proportional to the commutator $[\hat{f}_1, \hat{f}_2]$ for any two functions $f_1 , f_2$. 

The Hilbert space of prequantization is usually too huge for most purposes. 
Geometric quantization involves  construction of a polarization of the symplectic manifold such that we now take polarized sections of the line bundle, yielding a finite 
dimensional Hilbert space in most cases. However, $\hat{f}$ does not 
map the polarized Hilbert space to the polarized Hilbert space  in general. Thus only a few observables from the set of all  $f \in C^{\infty}(M)$ are quantizable.

When $M$ is a compact  K\"{a}hler manifold with  $\omega$  an integral  K\"{a}hler form and ${\mathcal L}$ the prequantum line bundle,   one can take as the Hilbert space of 
quantization the space of holomorphic sections of ${\mathcal L}^{\otimes \mu}$ for $\mu \in {\mathbb Z}$ large
enough. See ~\cite{W} for example for an explanation. We call such a  ${\mathcal L}^{\otimes \mu}$ a quantum bundle.

\vspace{.2in}

The determinant line  bundle was originally constructed by Knusden and Mumford and later on generalized by Quillen  as the determinant line bundle of a family of Cauchy Riemann operators on the space of connections on certain vector bundles on a compact Riemann surface~\cite{Q}. There are many generalizations, see  Bismut and Freed, ~\cite{BF} and in Bismut, Gillet, Soul$\acute{\rm{e}}$~\cite{BGS}. 

In the notation of Bismut, Gillet, Soul$\acute{\rm{e}}$~\cite{BGS}, one can take $B$ to be the moduli space of connections of a vector bundle on  a compact Riemann surface $\Sigma$, $Z$ to be $\Sigma$  and $M$ to be the trivial product $M = Z \times B$. Define the determinant of a family of Cauchy Riemann operators on $\Sigma$ parametrized by $B$, the space of connections. This yields a Quillen bundle in the  definition  according to ~\cite{Q}.

We will denote a bundle as a determinant bundle if  the bundle or its  rational power is a bundle in the sense of Quillen ~\cite{Q} or Bismut and  Freed ~\cite{BF} or Bismut, Gillet, Soul$\acute{\rm{e}}$  ~\cite{BGS}.


Let $(X, \omega)$ be a compact  integral K\"{a}hler manifold such that $L$ is a holomorphic positive  line bundle whose curvature is 
proportional to the K\"{a}hler form $\omega$.   Since for  $\mu > \mu_0 > 0$ for some $\mu_0$  the holomorphic bundle $L^{\otimes \mu}$ has enough holomorphic  sections for an embedding in projective space it can be considered as a quantum bundle.  The Hilbert space 
of quantization will  be the space of holomorphic sections
of $L^{\otimes \mu}$. It should be noted that since  $X$ is compact square integrability of holomorphic sections follows automatically.  

An interesting question is whether  this bundle $L$ or $L^{\otimes \mu}$ can be thought of  as a determinant bundle. This is the main thrust of this paper.

In ~\cite{DM}, Dey and Mathai had shown that one can realize a certain tensor product of $L$ as a Quillen bundle, i.e. determinant  of a family of Cauchy Riemann 
 operators on ${\mathbb C}P^1$, parametrized by $X$. Even when $(X, \omega)$ is just compact and symplectic, with integral symplectic form, Gromov embedding theorem enables us to embed $X$ into ${\mathbb C}P^N$ and they 
had shown in ~\cite{DM} that the quantum bundle is again of Quillen 
type, i.e. determinant  of a family of Cauchy Riemann operators on ${\mathbb C}P^1$, parametrized by $X$.

 In ~\cite{Bi}, Biswas proved an equivariant version of the results in ~\cite{DM} where the Cauchy Riemann operators can be on any compact Riemann surface (not necessarily ${\mathbb C}P^1$). 
 
In ~\cite{D1},~\cite{D2}, ~\cite{DP} Dey and in ~\cite{ER} Romao and Eriksson have constructed Quillen bundles on the moduli space of vortices on a Riemann surface.

In ~\cite{D3}, ~\cite{D4}  Dey has constructed Quillen bundles on  the moduli space of Higgs bundles for the three 
K\"{a}hler forms mentioned in the paper by Hitchin ~\cite{H} which constitute a hyperK\"{a}hler structure.

Donaldson and Kronheimer have given an exposition of the Quillen construction over the moduli space of ASD connections on a Riemann surface, see for instance $(6.5.4)$ in ~\cite{DK}.

 Motivated by these examples, in this paper we have shown that on projective manifolds $M$, the geometric quantum bundle can be realized in many ways as a  Quillen bundle in the  
sense of Bismut, Gillet, Soul$\acute{\rm{e}}$, namely one for each  input data of a Hermitian holomorphic line bundle $L$ on  a compact K$\ddot{\rm{a}}$hler manifold $Z$  and a geometric quantization of $(M, \omega)$. 

 The vortex equations when 
 defined on a K$\ddot{\rm{a}}$hler surface $(X, \omega)$ yields a non-empty  moduli space under certain conditions, see  Bradlow ~\cite{B1}. In fact in this  case, the moduli space is compact,  as it can be  described as a Seiberg -Witten moduli space (see Appendix). and the regular part of the moduli space is K$\ddot{\rm{a}}$hler.  Denoting the  K$\ddot{\rm{a}}$hler form to be $\Omega$  and if
 the Poincar$\acute{\rm{e}}$ dual to $\omega$ (the  K$\ddot{\rm{a}}$hler  form of the  K$\ddot{\rm{a}}$hler surface)  has a representative which is a Riemann surface (one dimensional
complex submanifold), then there
is a pullback of a Quillen bundle on the regular part of the  moduli space whose curvature is proportional to $\Omega$ (provided
a mild obstruction is bypassed).
As an application we show that for projective $4$- manifolds, the vortex moduli space is projective ( if the moduli space  is smooth or regular).

 In Appendix  we recall that there is a connection with Seiberg-Witten moduli space and the vortex moduli space for compact K\"{a}hler $4$-manifolds, 
 as showed by Bradlow and Garcia-Prada~\cite{BG}.

\section{{\bf Quantum  line  bundle as a determinant bundle}}

We begin by mentioning a well known fact  namely, for $Z$ a Riemann surface and $M = {\mathcal A}$ , the space of unitary connections on a vector bundle $E$ on the 
Riemann surface, the the three  constructions, namely the  one of Bismut and Freed ~\cite{BF} and Bismut, Gillet and Soul$\acute{\rm{e}}$, ~\cite{BGS}
and the one of Quillen ~\cite{Q}, matches. The equivalence of the first case and third case has been in disussed in the introduction of ~\cite{BF} and proved in details from section $f$ to section $k$. The equivalence of  all three  has also been dicussed in the introduction of \cite{BGS} and the proof is discussed in an series of related papers. The interested reader may look into these as it is beyond our scope of our paper to describe them fully.    

 In Quillen's case, the family of $\bar{\partial}$ operators on a Riemann surface $Z$ is parametrized by ${\mathcal A}$ where ${\mathcal A}$ is the space of unitary connections  of a vector bundle $E$ on $Z$. The $\bar{\partial}_A$ acts on sections of the vector bundle $E$ for $A \in {\mathcal A}$. 

 We have  a following:

 {\bf Nomenclature:}
For purposes of this paper, we shall use the term    determinant line bundle if the line bundle in question is  an rational power of a Quillen bundle or a bundle in the sense of Knusden-Mumford as in  Bismut and Freed ~\cite{BF} and Bismut, Gillet and Soul$\acute{\rm{e}}$, ~\cite{BGS}. We will still use the term  Quiilen bundle if the bundle is as in ~\cite{Q} with a modified Quillen metric.

We recall 
 
 \begin{defn}
 
 A  holomorphic line bundle  equipped with a connection and Hermitian metric on a compact  K$\ddot{\rm{a}}$hler  manifold with integral K$\ddot{\rm{a}}$hler form is called quantum  
 line bundle if the curvature of the line bundle (w.r.t. the connection) is proportional to
  the  K$\ddot{\rm{a}}$hler  form with the appropriate proportionality constant ( which can be positive or negative ). In this case the Hilbert space consists of square integrable 
  polarized sections with polarization
  by holomorphic tangent space for  positive constant or antiholomorphic for negative constant.
 \end{defn}.

 We have the following theorem:
 
 \begin{theorem} 
 If $M$ is  a compact K$\ddot{\rm{a}}$hler manifold with integral K$\ddot{\rm{a}}$hler form. Then the quantum bundle, or its non-zero power, positive or negative, 
 is holomorphically equivalent to  a determinant line  bundle ( modulo tensoring with a flat bundle), thus making the determinant bundle also a quantum  line bundle . There are 
 choices i.e. one for each choice of 
 a     compact K$\ddot{\rm{a}}$hler manifold $Z$ with non-zero Todd genus and having a holomorphic Hermitian bundle $L$ with nontrivial Chern class. (We may have to replace $L$ with $L^{\otimes n}$ for some $n$). 
 \end{theorem}

\begin{proof}
Let ${\mathcal L}$ be a quantum   line bundle on $M$.

Let $Z$ be a compact  K$\ddot{\rm{a}}$hler manifold  with Todd genus $ \neq 0 $ with a holomorphic line bundle $ L$ on it  with non-trivial Chern class.

Let  $Y= Z \times M$.  Let ${\mathcal L}$ be a quantum
line bundle on $M$ and $L$ a holomorphic line bundle on $Z$. Then  define $\zeta={p_1}^{*}(L) \otimes {p_2}^{*}({\mathcal L})$. Let $Q$ be the determinant line bundle  according to Bismut, Gillet, Soul$\acute{\rm{e}}$ ~\cite{BGS} corresponding to
$\zeta$. 

According to Bismut, Gillet and Soul$\acute{\rm{e}}$, the    curvature  $Q_{\zeta}$ of 
 $Q$ is given  by the degree two term of $2i\pi\int_Z Td(\frac{-R ^{Z}}{2i\pi}) Tr [exp(\frac{\Omega_{\zeta}}{2i\pi})]$
where $R^{Z}$ is the curvature of $T_Z^{0,1}$ and $\Omega_{\zeta}$  is the  curvature of the holomorphic Hermitian connection on $\zeta$. 

Note that 
 $\Omega_{\zeta} = {\mathcal K}_{{\mathcal L}} + {\mathcal K}_{L}$ where ${\mathcal K}_{{\mathcal L}}$ and
${\mathcal K}_{L}$ are the curvatures of ${\mathcal L}$ and $L$. In our case $R^{Z}$ and ${\mathcal K}_{L}$ are absorbed in the integral  leaving only an   constant multiple of 
 ${\mathcal K}_{{\mathcal L}}$ as the degree two part of the above expression. Thus   curvature $Q_{\zeta}$ is thus proportional to   ${\mathcal K}_{{\mathcal L}}$
 with a 
 rational proportionality constant.  Suppose the constant is strictly positive. In this case the determinant bundle is  a positive rational  tensor power of ${\mathcal L}$ 
 . 
 If the constant is strictly negative, the determinant  bundle is  a positive rational tensor power of   ${\mathcal L}^{-1}$modulo tensoring with a flat bundle.

 This proportionality constant may be zero. Then $Q_{\zeta}$ cannot be the curvature  to a quantum bundle or its tensor product since symplectic forms on compact manifolds have non trivial cohomology. In this case since  the Chern class
 of our line bundle $L$   on $Z$ is  non trivial and $Z$ has non-zero Todd genus, replacing $L$ with   $ \tilde{L} = L^{\otimes n}$  for an appropriate   $n$ in the above expression, we will get a non-zero constant. This can be seen as follows. The constant of proportionailty  looks like $\int_Z ( Td )  + n d_1 + n^2  d_2 + ... $ (a terminating polynomial) where $d_i$ are constants. Here $Td$ denotes $ Td(\frac{-R ^{Z}}{2i\pi}) \neq 0$.  We get a non-identically zero polynomial in $n$.  Since a polynomial has finite numer of zeros there will be some $n$ for which the constant  is non zero.

 This implies that the determinant  bundle and rational powers (positive or negative) of the quantum ample bundle differ by  atmost by a holomorphic bundle which is flat . So we can quantize by  a determinant line bundle. Since the determinant
 bundle and the quantum bundle we started with and its rational powers may differ by flat bundles they may not be holomorphically equivalent, so  we may hope to get new quantization.
 
\end{proof}

 The above theorem says that  under the conditions above, we can quantize $M$ by the determinant  bundle.

A similar result for  the abelian  vortex moduli space on a Riemann surface, Dey,~\cite{D1},~\cite{D2}.

  One can define the vortex equations on  compact K$\ddot{\rm{a}}$hler $4$-manifolds. The moduli space was studied by Bradlow ~\cite{B1}.  We geometrically quantize the  moduli space   where the equations are  
  now defined on a compact   K$\ddot{\rm{a}}$hler surfaces.

\section{ {\bf Vortex equations on a Riemann surface, Quillen bundle  and curvature}}

We review some material from ~\cite{D1}, ~\cite{D2}, ~\cite{DP}.

On  a Riemann surface $M$ the abelian vortex equation are given by
\begin{equation}\label{vor1}
\bar{\partial}_{A} \phi= 0,
\end{equation}

\begin{equation}\label{vor}
F_A = \frac{1}{2}(\tau  - {\lvert \phi \rvert }^{2})\omega.
\end{equation}

 where $\phi$ is defined as a smooth section  of an Hermitian holomorphic line bundle $L$ and $A$ is the unitary connection of the principal bundle $P$ associated to $L$. 
The form  $\omega$ is an imaginary valued   K$\ddot{\rm{a}}$hler form on $M$ and $\tau$ a real constant.

\subsection {The symplectic form }
Let $\mathcal{A}$ be the space of all unitary connections on P the associated principal bundle of the vortex bundle $L$ and $\Gamma(M, L)$ be sections of L. 
We define the configuration space as $\mathcal{C} = \mathcal{A} \times \Gamma(M, L)$. The space  $\mathcal{C}$ is an infinite dimensioanal affine space with differentail structure determined by tangent spaces whose tangent vectors are of the form $(\alpha,\phi)$ where $\alpha$ is
a $\mathfrak{u}(1)$ valued one form and $\phi$ a section.
Let  $p = (A, \Psi) \in \mathcal{C}, X = (\alpha_1 ,\beta), Y = (\alpha_2 ,\eta ) \in  T_p \mathcal{C}$. We define the following $L^2$-metric on $\mathcal{C}$.

\begin{equation}
\mathcal{G}(X,Y)= \int_M *\alpha_1 \wedge \alpha_2  +i\int_M (\frac{\beta \bar{\eta} + \bar{\beta} \eta}{2}) \omega
\end{equation}
and an almost complex structure $I=(*,i)$  on $T_p \mathcal{C}$ where $*dz_1=-i dz_1$ and $*d\bar{z_1}=i d\bar{z_1}$.

We define 

\begin{equation}
\Omega(X,Y)= -\int_M \alpha_1 \wedge \alpha_2  -\frac{1}{2}\int_M (\beta {\bar{\eta}} -{\bar{\beta}} \eta) \omega
\end{equation}
such that $\mathcal{G}(I X, Y ) = \Omega(X, Y )$.

Let $\zeta \in Maps(M, \mathfrak{u}(1))$ be  an element of the  Lie algebra of the gauge group. Note that ${\bar{\zeta}} =-\zeta$. It generates a vector field $X_{\zeta}$ on $\mathcal{C}$ as follows:
\begin{equation}
X_{\zeta} (A, \phi) = (d \zeta, -\zeta \phi) \in T_p\mathcal{C},
\end{equation}
where $p = (A, \phi) \in  \mathcal{C}$.
We show next that $X_{\zeta}$ is Hamiltonian. Namely, define $H_{\zeta} : \mathcal{C} \rightarrow \mathbb{C}$ as follows:
\begin{equation}
H_{\zeta} (p) =\int_M \zeta.(F_A -\frac{1}{2}(\tau -{\lvert{\phi} \rvert}^{2})\omega)
\end{equation}

Then for $X = (\alpha, \beta) \in T_p{ \mathcal C})$,

\begin{equation}
dH_{\zeta}(X)= -\int_M d\zeta\wedge \alpha - \frac{1}{2} \int_M {\bar{\beta}}(-\zeta)\phi -{\beta} (\zeta) {\bar{\phi}}) \omega
\end{equation}
 \begin{equation}
  dH_{\zeta}(X)= \Omega (X_{\zeta},X)  
 \end{equation}

Thus we can define the moment map $\mu : {\mathcal C} \rightarrow  \Omega^{2}(M, \mathfrak{u}(n)) = {\mathfrak{g}}^{*}$ (the dual of the Lie
algebra of the gauge group) to be
\begin{equation}
\mu(A, \phi) = F_A - \frac{1}{2}(\tau - {\lvert \phi \rvert}^{2} )\omega
\end{equation}

The moduli space of the vortex equations \eqref{vor1} and \eqref{vor} is defined as the quotient of the space
of solutions  by the gauge grooup. It inherts its topology
by the quotient topology and the differential structure
by  taking quotient of the 
space of tangent vectors of the solution space by the gauge group.
By K$\ddot{\rm{a}}$hler reduction(see 2.3,\cite{ER}) this form  descends to the moduli space, giving it a K$\ddot{\rm{a}}$hler structure. In fact it is related to the Manton-Nasir form ~\cite{MN}.

  {\bf Note:} In this section we had taken $\omega = h^2 d z \wedge d \bar{z}$,  an imaginary valued  symplectic form on the Riemann surface which is also K$\ddot{\rm{a}}$hler. If instead we take it to be the symplectic form on the real tangent space, namely,  $\tilde{\omega} = i \omega$, then the 
 second  vortex equation looks like:
  
  \begin{equation}\label{vor2}
F_A = \frac{i }{2}( |\phi|^2- \tau  )\tilde{\omega}.
\end{equation}

This is the form in which the equation been written in  ~\cite{B1}.

\subsection{The modified Quillen metric and curvature}\label{QMC}

In ~\cite{MN}, it was discussed that the Manton-Nasir-Samols form on the vortex moduli space is integral when the   Riemann surface has volume an integral
multiple of $ \frac{4\pi}{\tau}$ (see ~\ref{mn_form}).

Let ${\rm det} (\bar{\partial})$ denote the Quillen bundle defined on ${\mathcal A}$ as in ~\cite{Q}. 
Let $pr: {\mathcal C} = {\mathcal A} \times \Gamma(L) \rightarrow {\mathcal A}$.
We denote the Quillen bundle ${\mathcal P} = pr^*({\rm det} (\bar{\partial}))$ which is well defined on 
${\mathcal C}={\mathcal A} \times \Gamma (L)$ which is an affine space.
We can  equip  ${\mathcal P}$ a modified Quillen metric, namely, we multiply the 
Quillen metric ~\cite{Q} by the factor  $e^{- \frac{i}{4 \pi} \int_{M} | \Psi|^2_H \omega} $.

 The Quillen metric contributes $-\frac{i}{2\pi} \left( \int_{\Sigma} \alpha_1 \wedge \alpha_2\right)$ to the curvature 
 ~\cite{Q}, 
and the factor $e^{- \frac{i}{4 \pi} \int_{M} | \Psi|^2_H \omega}$ contributes
$\frac{i}{2\pi} \left(- \frac{1}{2}\int_{\Sigma}( \beta  \bar{\eta} - \bar{\beta}  \eta ) \omega\right)$  to the curvature. 

Thus we have the following:

The curvature of ${\mathcal P}$ with the modified Quillen metric is indeed  $\frac{i}{ 2 \pi} \Omega$ on the affine space ${\mathcal C}$.

$\Omega$ descends to the moduli space as a K$\ddot{\rm{a}}$hler form by K$\ddot{\rm{a}}$hler reduction. 


If  descent of $\Omega$ is integral on th moduli the 
Quillen bundle ${\mathcal P}$ on the configuation space  descends to the moduli space. This follows from considerations in ~\cite{W} (chapter on prequantization)  and ~\cite{D1}, ~\cite{D2}, ~\cite{DP} etc.

Let $\Omega_{MN}$ be the Manton-Nasir form as defined in ~\cite{MN}.  Then $\Omega_{Samols} =  \Omega_{MN}$.  (see ~\cite{MN} equaton 2.16, arxiv version and discussion aboove the equation). 

It can be shown that $[\frac{\Omega}{2 \pi}] = [\Omega_{Samols}] =\Omega_{MN} $, see for instance ~\cite{DM}.

 The  condition that the bundle descends (i.e. $\Omega$ is integral on the moduli space) is that the Riemann surface has volume $A$ an integral multiple of $\frac{4 \pi}{\tau}$.   This is because 

\begin{equation}\label{mn_form}
[\frac{\Omega}{2\pi}]=  [\Omega_{MN}] = [(\frac{\tau A}{2}-2\pi N)\eta +2\pi(\sigma_1 +\ldots+ \sigma_n)]
\end{equation}
 where $A$ is the volume of the Riemann surface $\Sigma$ and $\sigma_i's$ and $\eta$ are integral cohomological classes of the moduli space (defined in~\cite{MN}).

\section{{\bf Generalizations of above theory to the moduli space of vortices on a K$\ddot{\rm{a}}$hler surface }}
 For K$\ddot{\rm{a}}$hler sufaces $X$(i.e.  compact K$\ddot{\rm{a}}$hler 4 real dimensional manifolds)  vortex equations are as follows ~\cite{B1}.
 
 \begin{eqnarray}\label{kvor}
 \bar{\partial}_A \phi  &=& 0 \\
\Lambda F_{A} &=& \frac{i}{2} ({\lvert {\phi} \rvert}^2 -\tau)\\
 {F_A}^{0,2} &=& 0
 \end{eqnarray}
where $\Lambda F_{A}$ is the contraction of $F_A$ with a suitable  K$\ddot{\rm{a}}$hler form $\omega$ (i.e. the symplectic form on the real tangent space)
   where $F_A$ is the curvature of a line bundle $L$  on $X$ with connection $A$, $\phi$ is a section.
 
Let  the configuration space be  ${\mathcal C}= \mathcal{A} \times \Gamma(L)$, where $\mathcal{A}$ is the affine space of unitary  
connections on $L$ and $\Gamma(L)$ is the space of sections of $L$. The space $\mathcal{C}$ is an infinite dimensional  affine space and its tangent space is similar
to the that discussed in the Riemann surface case discussed in previous section.

 \subsection{The moduli space as a  K$\ddot{\rm{a}}$hler  manifold}

 In the first part of this section we closely follow Riera ~\cite{R}.
 
 Let $\mathcal{A}$ be the space of $U(1)$-connections on $P$ the associated bundle of a line bundle $L$. This is an affine space modelled on
$\Omega^{1}(P \times_{Ad} \mathfrak{u}(1))$. We define a complex structure $I_{ \mathcal{A }}$ on A as follows. Given any $A \in \mathcal{A}$ , the
tangent space $T_A \mathcal{A}$ can be canonically identified with $\Omega^{1} (P \times_{Ad} \mathfrak{u}(1)) = \Omega^{0} (T^{*}(X) \otimes P \times_{Ad} \mathfrak{u}(1))$.
Then we set $I_\mathcal{A} = -I^{ *} \otimes 1$, where $I$ is the complex structure of the tangent bundle which it inherits  since the manifold is complex. The complex structure $I_\mathcal{A}$ is integrable. We also define on
$\mathcal{A}$ a symplectic form $\omega_{\mathcal{A}}$ . Let $\Lambda : \Omega^{p,q} (X) \rightarrow \Omega^{p-1,q-1} (X)$ be the adjoint of the map
given by wedging with $\omega$ which is also equal to the contraction with $\omega$ with respect to K$\ddot{\rm{a}}$hler metric . Then, if $A \in \mathcal{A}$ and 
$\alpha_1, \alpha_2 \in T_A \mathcal{A} = \Omega^{1} ( P \times_{Ad} \mathfrak{u}(1))$, we set
\begin{equation}\label{kah}
\omega_\mathcal{A} (\alpha_1, \alpha_2) = -\int_X \Lambda(B(\alpha_1,\alpha_2)) \frac{\omega \wedge \omega}{2} .
\end{equation}
Here $B : \Omega^{1} (P \times_{Ad} \mathfrak{u}(1)) \otimes \Omega^{1}(P \times_{Ad} \mathfrak{u}(1)) \rightarrow \Omega ^{2} (X)$ is the combination of the usual wedge
product with a bi-invariant nondegenerate pairing $< >$, on $\mathfrak{u}(1)$. Since  $\mathfrak{u}(1)$ is one dimensional  we can consider $B$ as wedge of imaginary one forms. 

It turns out that $ \omega_\mathcal{A}$ is
a symplectic form on $\mathcal{A}$ , and it is compatible with the complex structure $I_{\mathcal{A}}$ . Hence
$\mathcal{A}$ is a  K$\ddot{\rm{a}}$hler manifold. 
 Let $X _1= (\alpha_1, \beta) $ and $X_2 = (\alpha_2, \eta) $ are in $ T_{(A, \phi)} {\mathcal C}$.

On ${\mathcal C}$ we define 
\begin{equation}
\Omega_X(X_1, X_2) =- \int_X \Lambda(B(\alpha_1,\alpha_2))\frac{\omega \wedge \omega}{2} +\frac{i}{2} \int_X (\beta  \bar{\eta} - \bar{\beta} \eta)\frac{\omega \wedge \omega}{2}  .
\end{equation}
where $\omega$ is now a  K$\ddot{\rm{a}}$hler form(a symplectic form  compatible with complex structure and   K$\ddot{\rm{a}}$hler metric).  The form $\Omega_X$ is the K$\ddot{\rm{a}}$hler  form with respect to the following K$\ddot{\rm{a}}$hler metric.
\begin{equation}
g(X_1,X_2)= \int_M *\alpha_1 \wedge \alpha_2 \wedge \omega  +\int_M (\frac{\beta \bar{\eta} + \bar{\beta} \eta}{2}) \frac{\omega \wedge \omega}{2}
\end{equation}
where $*$ defined in preceding section.
There exists a moment map for the action of $\mathcal{G}_{U(1)}$ on $\mathcal{A}$ , which takes the following
form (see for example \cite{DK}, \cite{Ko}):

$\mu : A \rightarrow Lie{\mathcal{G}^*_{U(1)}}$

 $A \rightarrow \Lambda(F_A)$ .

 Here $F_A$ denotes the curvature of $A$. It lies in $\Omega^{2}(P \times_{Ad} \mathfrak{u}(1))$, so $\Lambda(F_A) \in \Omega^{0}(P \times_{Ad}  \mathfrak{u}(1)) ) \subset
{\Omega}^{0} (P \times_{Ad} \mathfrak{u}(1)))^{*}$, the last inclusion being given by the integral on $X$ of the pairing $<>$,  on $\mathfrak{u}(1)$.



Let $\mu(A, \phi ) = \Lambda F_{A}- \frac{i}{2} ({\lvert {\phi} \rvert}^2- \tau) $

 Following the same arguments from previous section we have
  $\mu = \mu(A, \phi) $ above   is a moment map for the action of the gauge group on ${\mathcal C}$.
 The moduli space inherits the quotient topology by the gauge group action and the differential structure is simmilar to the one discussed in section 3 in the Riemann surface case.

 Since the moduli  space is same as the  Seiberg-Witten 
moduli by Appendix,
     the regular part of the moduli space is  K$\ddot{\rm{a}}$hler with K$\ddot{\rm{a}}$hler form $\Omega_X$ by corollary 4.2 Becker~\cite{B} mainly by infitedimensiinal  K$\ddot{\rm{a}}$hler  reduction technique.
     
     The above argument works for K$\ddot{\rm{a}}$hler  surfaces for more general K$\ddot{\rm{a}}$hler  manifolds it has been shown in \cite{R}.

\subsection{Dterminant bundle construction on the moduli space}

The construction is similar to Donaldson's construction of the Quillen bundle on the moduli space of  ASD connections on a K$\ddot{\rm{a}}$hler surface~\cite{DK}, section $(6.5.4)$.

 Let $A \in {\mathcal A}$. Let us restrict the connection $A$ to the Poincar$\acute{\rm{e}}$ dual $S$ of the K$\ddot{\rm{a}}$ hler form $\omega$ on $X$. Let the 
 restricted connection be denoted by $A^{R}$ and the space of restricted connection be deonted by ${\mathcal A}^R$.  
  
 Let us consider the first term of the  symplectic form. For the K$\ddot{\rm{a}}$hler surface the  symplectic form $\Omega_{\mathcal{A}}$ is from \eqref{kah}
\begin{equation}
\Omega_{\mathcal{A}} (\alpha_1,\alpha_2) =  -\int_X (\Lambda(B(\alpha_1,\alpha_2)) \frac{\omega \wedge \omega}{2}.
\end{equation}
  $\Lambda$ is the contraction with respect to the K$\ddot{\rm{a}}$hler form $\omega$ and $B$ is wedging of $\alpha_1$ and $\alpha_2$. So we get 
\begin{equation}
\Omega_{\mathcal{A}} (\alpha_1,\alpha_2) = - \int_X \alpha_1 \wedge \alpha_2  \wedge \omega.
\end{equation}

 Recall  the closed K$\ddot{\rm{a}}$hler form $\omega$ belongs to a cohomology class $[\omega]$
 and thus
 there exists a homology class $D$ which is the Poincar$\acute{\rm{e}}$ dual of the cohomology class $[\omega]$. If we take submanifold representative of $D$, say $S$,
we have 
\begin{equation}
\int_X (\alpha_1 \wedge \alpha_2) \wedge \omega.= \int_S \alpha_1 \wedge \alpha_2 =-\Omega_R(\alpha_1, \alpha_2)
\end{equation}
for $\alpha_1 \wedge \alpha_2$ closed, where we denote $\Omega_R(\alpha_1, \alpha_2) = -\int_S \alpha_1 \wedge \alpha_2$

Let the Poincare dual of $\omega$ be $S$.
\begin{theorem}\label{dual}
 If $S$ is a Riemann surface i.e. it is a complex one dimensional submanifold of the K$\ddot{a}$hler surface $X$, then we  have a pullback of a determinant  line   
 bundle in the sense of Quillen
 with a metric on the vortex  configuration space.  Its curvature form is   cohomologus to the standard K$\ddot{\rm{a}}$hler
 form on the configuration space (with the appropriate proportionality factor)  and their  difference is given by the differential of a gauge invariant one form. Moreover if 
 this bundle  descends to the vortex moduli space,  its curvature with respect to a certain connection will be  proportional to the K$\ddot{\rm{a}}$hler form obtained
 by the moment map reduction of the K$\ddot{\rm{a}}$hler form to the regular part of the vortex  moduli space.
\end{theorem}

\begin{proof}
We assume $S$ is connected, though the proof goes through otherwise as well.
 We observe the restriction of  forms on $X$ will give a  map from     $\mathcal{A}$ to $\mathcal{A}^{R}$ the connection space of the
 Riemann surface
 $S$. We can pullback   the  Quillen bundle over $\mathcal{A}^{R}$  with curvature  proportional to
 $\Omega_R(\alpha_1,\alpha_1)=-\int_S \alpha_1 \wedge \alpha_2$ to a holomorphic bundle( since S is a Riemann surface by the above assumption). We claim that if the 
 pullback of this  bundle  descends on ${\mathcal A} / {\mathcal G}$ its curvature proportional to 
 \begin{equation}\label{modu}
 - \int_X ({\alpha}_1  \wedge {\alpha}_2 ) \wedge \omega= \Omega_{\mathcal{A}} ({\alpha}_1 ,{\alpha}_2 )
 \end{equation}
 
 The above statement follows verbatim from  section $(6.5.4)$, ~\cite{DK}.
 

 
 On ${\mathcal C} = {\mathcal A} \times \Gamma(L)$ we defined  the K$\ddot{\rm{a}}$hler
 form  which can now be written as 
 $    \Omega_X(X_1, X_2) =   -\int_X \alpha_1 \wedge \alpha_2 \wedge \omega + \frac{i}{2}\int_X (\beta  \bar{\eta} - \bar{\beta}  \eta) \frac{\omega \wedge \omega}{2}. $
 where $X_1 =(\alpha_1, \beta)$ and $X_2 = (\alpha_2, \eta)$ as before.
  By discussion in the previous subsection the regular part of the moduli space is K$\ddot{\rm{a}}$hler with K$\ddot{\rm{a}}$hler form the descendant  of the
above form $\Omega_X$.

 The    Quillen bundle is  the standard Quillen bundle on the Riemann surface, namely ${\rm det}(\bar{\partial}_{A^R})$ on the Riemann surface configuration space . 
 The pullback bundle induced by the restriction map with 
 metric ~\cite{Q} is modified with the factor $e^{\frac{i}{4 \pi} \int_X |\phi|^2 \frac{\omega \wedge  \omega}{2}}$ as in the Riemann surface case~\cite{D2}.

The conditions for descent of the line bundle will be discussed in the next subsection.
If this line bundle descends as in the Riemann surface case, the  cohomology of the curvature is proportional to the cohomology of the  descendant of  of $\Omega_X$ by \eqref{coho_1} since 
it can be checked as in~\cite{DK}
\begin{equation}\label{coho_1}
 \Omega_{\mathcal{A}} ({\alpha}_1 ,{\alpha}_2 )= i^{*}(\Omega_R)(\alpha_1,\alpha_2) +d(\Phi)(\alpha_1,\alpha_2)
\end{equation}
where notation is as in ~\cite{DK} and exactness of the second integrands of the two forms $\Omega_X$ and $2\pi i$ times curvature of the pullback of modified Quillen metric on the Riemann surface moduli.
From the above equation since $\Phi$ and $d(\Phi)$ is gauge invariant we get the statement above \eqref{modu}. (Here $i$ is the restriction map).

\end{proof}
\subsection{Integrality}
Let $X_1=(\alpha_1, \beta) $ and $X_2 =(\alpha_2, \eta)$ be tangent to the configuration space.
The form $\Omega_X$ given by
\begin{equation}
 \Omega_X(X_1,X_2)= -\int_X \alpha_1 \wedge \alpha_2  \wedge \omega  + \frac{i}{2}\int_X (\beta {\bar{\eta}} -{\bar{\beta}} \eta) \frac{\omega \wedge \omega}{2} 
\end{equation}
on the K$\ddot{\rm{a}}$hler surface  configuration and the solution subspaces. It  may not descend to an integral 
form on the moduli space. Same holds for corresponding form $\Omega_S$  for the moduli space in the the Riemann surface case.  Deriving their result from the Samol's metric, Manton  and Nasir gave 
a cohomological description of $\frac{1}{2\pi}$ times the form on  the vortex moduli space for a Riemann surface. Now since $\frac{\Omega_S}{2\pi} = \Omega_{Samols} = \Omega_{MN}$,from the discussion in subsection \ref{QMC},  
the cohomology 
class of $\frac{1}{2\pi}$ times  the form as in \cite{ER} is
\begin{equation}\label{mn_form1}
[\frac{\Omega_{S(A)}}{2\pi}]=[(\frac{\tau A}{2}-2\pi N)\eta +2\pi(\sigma_1 +\ldots+ \sigma_n)]
\end{equation}
where $A$ is the area of the Riemann surface $S$ and $\sigma_i's$ and $\eta$ are integral cohomological classes of the moduli space (defined in~\cite{MN}). From  
equation \ref{mn_form1} it 
 is clear
 $\frac{\Omega_{S(A)}}{4{\pi}^{2}}$ is integral if $\frac{\tau A}{4\pi}= n $  where $n$ is an integer. 
 
 Taking the form $\omega_1=k\omega$ where $k= n \frac{4\pi}{\tau A}$ we  get an integral K$\ddot{\rm{a}}$hler form 
  on the  moduli space for the Riemann surface.
 Using $\omega_1$ instead of $\omega$ we get an ample line bundle on moduli space of vortices on the Riemann surface $S$. 

We first define $\Psi$, a  holomorphic map between the vortex moduli space ${\mathcal M}_X$ for the  K$\ddot{\rm{a}}$hler surface $X$ and the vortex moduli  space ${\mathcal  M}_S$ 
for the  Riemann surface $S$ (with vortex equation with $\tau$ large enough such the moduli space is non-empty~\cite{B1}).
The pullback of the holomorphic line bundle by $\Psi$  will be a 
holomorphic line bundle on the moduli space of vortices for the K$\ddot{\rm{a}}$hler surface  and we will show the pullback  form defining the first Chern class is cohomologus 
to $-\frac{\Omega_X(\omega_1)}{4{\pi}^{2} }$ where $ \Omega_X(\omega_1)$ is the form $\Omega_X$ with $\omega$ replaced by $\omega_1$.

{\bf Obstructions O(1) and O(2)}:
Let the vortex line bundle and the Poincar$\acute{\rm{e}}$ dual $S$ to $\omega$ be such that a non-identically-zero holomorphic section does  not vanish entirely on $S$ or a component of
$S$ when $S$ is disconnected, i.e. 
the Chern class of the vortex line bundle is such that the Poincar$\acute{\rm{e}}$ dual does not contain a component of $S$.
If this condition is not satisfied it is called obstruction O(1).

When every representative of the Poincar$\acute{\rm{e}}$ dual to $\omega$ is not a complex submanifold of $X$, we call it  obstruction O(2).

\begin{lemma}\label{hol}
When obstructions $O_1$ and  $O(2)$ is not satisfied, there is a holomorphic map $\Psi$ between the vortex moduli space ${\mathcal M}_X$ for the K$\ddot{\rm{a}}$hler surface $X$ and the 
vortex moduli space ${\mathcal M}_S$ for  the Riemann surface $S$.
\end{lemma}

\begin{proof}
 Here we again assume $S$ is connected.
  Let us define $\Psi:  {\mathcal M}_X \rightarrow {\mathcal M}_S $ by $\Psi([A, \phi)] = [A_{R}, \phi_R]$ 
 where $\phi_R$ denotes restriction of $\phi$ to $S$ and $A_R$ is the only connection (see \cite{B1} sections $3$ and $4$) which is a solution to the 
 vortex equation on the Riemann surface $S$ with section  $\phi_R$ and holomorphic structure the restriction of the holomorphic structure due to $A$. Since the first vortex equation of  $X$ restricts to the first vortex equation of $S$ as the later being a complex one dimensional submanifold of $X$ the restriction of section and holomorphic structure works out.
 
  Orbits map to orbits: If two elements $(A,\phi)$ and $(A_1, \phi_1)$ are related by gauge transformations  in the K$\ddot{\rm{a}}$hler surface configuration space then 
  the section part restrictions are related  by restriction of gauge transformations, and the connection parts will be related by the same gauge transforms by uniqueness and 
  since the equations are gauge invariant.
 Thus the map $\Psi$ is well-defined.
 
  Map between the moduli spaces:  From a close inspection of \cite{B1} one can see the moduli of solutions ${\mathcal M}_X$ is given by effective divisors. For an effective divisor in the  K$\ddot{\rm{a}}$hler surface $X$ we
 can take a holomorphic section and holomorphic  structure as a representative and  its restriction on the Riemann surface $S$ (which is a one dimensional complex submanifold)
 will define a divisor on the Riemann surface provided the restriction is not identically zero which is avoided since obstruction $O_1$ is not satisfied.
  This coincides with the map $\Psi$ between between moduli spaces which we defined in the first line. Though our connection part may not agree with the restriction of  a connection of 
  a  K$\ddot{\rm{a}}$hler surface moduli but the holomorphic structure do match because of the first vortex equation and $S$ being a complex one dimensional submanifold.

  Holomorphicity of the map $\Psi$: The map between the moduli is holomorphic as the Riemann surface $S$ is a complex one dimensional submanifold of $X$ and the section  part is just the restriction . Since  acting the section part of a tangent vector of the  space of solutions by the  complex structure of the section part changes the connection part of the vector by  its  complex structure ( since the moduli of solutions are complex analytic) the map is holomorphic.
  \end{proof}
 
 Since we chose $k= n \frac{4\pi}{\tau A}$ ( where $n$ an integer and $A$ area of $S$ induced by $\omega$), the negative of the K$\ddot{\rm{a}}$hler form 
 $\frac{\Omega_{S(A_{\omega_1})}}{4 {\pi}^{2} }$  on the 
 moduli space is 
 integral ( $A_{\omega_1}$ is the the area of $S$ with volume form $\omega_1 = k \omega$) and hence  its cohomology class will be a Chern class of a holomorphic 
 line bundle ${\mathcal L}$ and so its pullback by $\Psi$ to the K$\ddot{\rm{a}}$hler surface moduli space ${\mathcal M}_X$ will be a holomorphic line bundle $\Psi^{*}({\mathcal L})$.
 
 \begin{lemma}
   The  cohomology class of the form  $-\frac{\Omega_X(\omega_1)}{4{\pi}^{2} }$   is the Chern class of $\Psi^{*}({\mathcal L})$.Thus from this we get a quantization of vortex moduli by determinant bundle in the sense mentioned in Section 2
 \end{lemma}
 \begin{proof}
 The Chern class of the bundle $\Psi^{*}(L) $ is the cohomology class of the pull back of  $-\frac{\Omega_{S(A_{\omega_1})}}{4 {\pi}^{2} }$. The 
form in the configuration space level is given by
\begin{equation}\label{expo}
\frac{ \Omega_{S(A_{\omega_1})}}{4{\pi}^{2}}=\frac{-\int_S \alpha_1 \wedge \alpha_2    + i \int_S (\beta {\bar{\eta}} -{\bar{\beta}} \eta) \frac{\omega_1}{2}}{4{\pi}^{2}}  
\end{equation}
Though our map is  not exactly the restriction   the holomorphic structure part  of $A_R$ and $A^{R}$
  (the actual restriction of $A$   where $[(A,\phi)]$ is a solution in $X$) are same($\bar{\partial}_{A_R}=\bar{\partial}_{A^{R}}$). The reason behind this is that the Riemann surface  $S$ is a one dimensional complex submanifold of $X$ and so the first vortex equation of $X$ restricts to the first vortex equation of $S$. Now since the fibre of the   Quillen determinant bundle as defined in \cite{Q} and the metric depends  on the holomorphic structure(the delbar part)    the pullback of $\Psi$ and  descent of the pull back bundle  in moduli level described in \ref{dual}    yield holomorphically  equivalent isometric bundles. Thus we get a determinant bundle since pullback of a determinant bundle is a determinant bundle( see introduction \cite{BGS}) with
 the required curvature from below.

  The following argument has been made in theorem \ref{dual} but we repeat for the reader's convenience. From     \eqref{coho_1} for the form $\omega_1$ we get the first integrand of the numerator of the right hand side of \eqref{expo} is cohomologus to  first integrand of $\Omega_X(\omega_1)$ by the correspondence in the above paragraph since $\Psi$ and the restriction map pullbacks produce isometric bundles. The 
  second   term
   of \eqref{expo} is $\frac{\int_S (i\beta {\bar{\eta}} - i{\bar{\beta}} \eta) \frac{\omega_1}{2}}{4{\pi}^{2}}$ is the differential of the  form  
   $\Psi_{\phi}(\eta)= -\frac{i\int_S \eta {\bar{\phi}}\omega_1}{4{\pi}^{2}}$ .
   Similarly it can be shown the form representing the second term of the  $\frac{\Omega_X(\omega_1)}{4{\pi}^{2} }$ is exact. Both cases the one forms are gauge invariant.
   So the difference of the two forms $ \Psi^{*}(\frac{\Omega_{S(A_{\omega_1})}}{4{\pi}^{2} })$ and$\frac{\Omega_X(\omega_1)}{4{\pi}^{2} }$  descend  to exact forms
making them cohomologus. Since moduli for the form 
$\omega_1 $ and the form $\omega$ is biholomorphic and the pull back of the K$\ddot{\rm{a}}$hler form of $\omega_1$ moduli is cohomologus to th k times a K$\ddot{\rm{a}}$hler  form in the $\omega$ moduli   by arguments as above, we get a quantization of the original  $\omega$ moduli by a determinant bundle. (the bundle ${\mathcal{L}}^{-1}$ is a Quillen bundle as $-$ times its curvature is the integral descent of the Quillen curvature  and since positive or negative rational  powers of  Quileen bundle are determeinant bundles by our convention (see section 2) and since determinant bundles are closed under the pullbacks our claim follws). 
\end{proof}

\subsection{Projectivity of moduli space}

  
     The above theory can be generalized to $S$ having  more than one connected components. In that case the  obstruction to getting a holomorphic bundle on the regular part 
     the whole 
     moduli space is existence of sections 
     whose zero sets contain a component of $S$.
  
    The moduli  space ${\mathcal M}_X$ of vortex equations   on a K$\ddot{\rm{a}}$hler surface $X$  is   compact since it is equivalent to the Seiberg Witten moduli. 
      
On the other hand, we proved that  the regular moduli  space ${\mathcal M}_X$   has a K$\ddot{\rm{a}}$hler form $\Omega_X$, which is integral under the 
the condition   that   obstructions O(1) and O(2) can be avoided.

Under this condition we have shown, 
there is a determinat line bundle whose curvature is proportional to $\Omega_X$.  
 Thus from results proved in previous subsection we have a 
quantization by a determinant bundle
of the vortex K$\ddot{\rm{a}}$hler surface moduli space. 
 
 Below we give a large class of manifolds for which the obstructions can be avoided. We mention the following proposition  which follows also from Bradlow's work~\cite{B1}. The 
 main result in Bradlow's paper  implies that the moduli space can be interpreted as a Hilbert scheme of hypersurfaces in the base (K$\ddot{\rm{a}}$hler, projective) manifold of a fixed degree (the degree of the line bundle). That such a Hilbert scheme is projective is a well-known fact (Grothendieck's EGA or FGA). This observation is due to N. Romao. 
 
 There may be other proofs  of the following proposition probably one due to  J.M.Baptista whose reference we are unable to provide.
 
 We
 mention this proposition as it follows easily  from the results of this section without going into the theory of Hilbert schemes. It also provides an example of the ample or the quantum bundle being  a determinant bundle and it may differ from the one shown in \cite{DM} as in our case the input bundle is on a K$\ddot{\rm{a}}$hler surface while there's was
 a bundle on $CP^{1}$.

\begin{proposition}
Let $X$ be a projective K$\ddot{\rm{a}}$hler surface  with integral K$\ddot{\rm{a}}$hler form $ \omega$. Then the  moduli space ${\mathcal M}$ of vortex equations on $X$
  is projective  if the moduli space is smooth.
 \end{proposition}

\begin{proof}
 Since the moduli space  is smooth and compact, we have to show that obstructions $O_1$ and $O_2$ are avoided in order to get an ample line bundle. Since the manifold is projective we can have a very ample bundle
 whose zero of a generic  section will make us avoid obstruction $O_2$. To avoid  obstruction $O_1$, suppose a bundle $L$  has an holomorphic section which is zero on divisor $D$ and if the divisor
 is smooth and irreducible.
 \begin{equation}
 deg(L)-deg(D) \geq 0
 \end{equation}  
      Let $\mathcal{L}$ be the ample bundle then ${\mathcal{L}}^{p}$ for large $p$ be a very ample bundle then by increasing $p$ we can have
  \begin{equation}
  deg(L)- deg({\mathcal{L}}^{p}) <0
  \end{equation}
 So the obstruction $O_1$ can be avoided for the vortex moduli space for $X$ with K$\ddot{\rm{a}}$hler  form  $p\omega$ if there is smooth and irreducible divisor of ${\mathcal{L}}^{p}$. But this
  is guaranteed by Bertini's theorem  since the bundle is very ample. For the vortex  moduli space corresponding to   $X$ with  K$\ddot{\rm{a}}$hler form $\omega$  this holds too (as the corresponding K$\ddot{\rm{a}}$hler  manifolds $(X, \omega)$ and $(X, p \omega)$ are biholomorphic and the pull back of the K$\ddot{\rm{a}}$hler form of $p \omega$ moduli is cohomologus to th $p$ times a K$\ddot{\rm{a}}$hler $\ddot{\rm{a}}$hler form in the $\omega$ moduli).

\end{proof}

\begin{corollary}
The vortex moduli space on the     K$\ddot{\rm{a}}$hler surface   is projective  if the moduli space  is smooth   and there exists a closed surface surface  S  which is a 
complex one dimensional 
submanifold whose homology class is the Poincar$\acute{\rm{e}}$ dual of the cohomology class of the  K$\ddot{\rm{a}}$hler form and none of the representatives of  the 
Poincar$\acute{\rm{e}}$ dual of the Chern class of the vortex line bundle  contain a component of S.
\end{corollary}

 \section{{\bf Appendix1 :  The Vortex  and  the Seiberg-Witten correspondence}}
In this section we briefly review the Seiberg-Witten equations for the  K$\ddot{\rm{a}}$hler surface and the analysis of these
equations in this  case. This section closely follows  Bradlow and Garcia-Prada,  \cite{BG} .

In \cite{BG}, Bradlow and Garcia-Prada   wrote the Seiberg-Witten equations  as
\begin{equation}
\bar{ \partial}_{\hat{A}} \phi + {\bar{\partial}_{\hat{A}}}^{*} \beta = 0
\end{equation}
\begin{equation}
\Lambda F_A = i( {\lvert \phi \rvert}^{2} - {\lvert \beta \rvert}^{2} )
\end{equation}
\begin{equation}
{F_A}^{2,0} = -\bar{\phi} \beta
\end{equation}
\begin{equation}
{F_A}^{0,2} = \bar{\beta} \phi
\end{equation}
where  notation is as in ~\cite{BG}, i.e. $(\phi, \beta)$ is a section of the $S_L^{+}$ bundle, $A$ is connection on $L$ and $\hat{A} $ is the induced connection on $\hat{L} $ and $\Lambda F_A$ is the contraction of the curvature with the K $\ddot{\rm{a}}$hler form $\omega$. 
It is not difficult to see (\cite{Wi}) that the solutions to these equations are such that
either $\beta= 0$ or $\phi = 0$, and it is not possible to have irreducible solutions of both types simultaneously for a fixed $Spin^{c}$ -structure. We thus
have one of the following two situations:

$(i)$ $ \beta = 0 $ and the equations reduce to

$F_A^{ 0,2} = 0$

$\overline{\partial_{\hat{A}}}\phi = 0$

$\Lambda F_A = i {\lvert \phi \rvert}^{2} $

and similar equations for 

$(ii)$ $\phi =0 $, i.e.

${F_A}^{0,2} = 0$

${\overline{\partial_{\hat{A}}}}^{*} \beta = 0$

$\Lambda{F_A} = i{\lvert \beta \rvert}^{2}.$

\begin{remark}. 
We have omitted the equation $F_A^{2,0} = 0$, since by unitarity of the connection this
is equivalent to $F_A^{0,2} = 0$.
\end{remark}
 The Hodge star operator
interchanges these two cases, and we can thus concentrate on case $\beta =0$.
Equations  are essentially the equations known as the vortex equations. These  have been extensively studied (e.g. in \cite{B1}, \cite{B2}, \cite{G2}, \cite{G3} ) for compact
K$\ddot{\rm{a}}$hler manifolds of arbitrary dimension. The equations are the following:
Let $(X, \omega)$ be a compact K$\ddot{{\rm{a}}}$hler manifold of arbitrary dimension, and let $(L, h)$ be a
Hermitian $C^{\infty}$ line bundle over $X$. Let $\tau \in R$. The $\tau$ -vortex equations

$F_A^{ 0,2} = 0$

$\overline{\partial_{\hat{A}}} \phi = 0$

$ \Lambda F_A = \frac{i}{2}({\lvert \phi \rvert}^{2} - \tau )$

are equations for a pair $(A, \phi)$ consisting of a connection on $(L, h)$ and a smooth section
of $L$. The first equation means that $A$ defines a holomorphic structure on $L$, while the
second says that $\phi$ must be holomorphic with respect to this holomorphic structure.

Let $s$ be  the scalar curvature of $X$, the Bradlow and Garcia-Prada ~\cite{BG} obtain that the Sieberg witten equations are equivalent to

$ {F_{\hat{A}}}^{0,2} = 0$

$ \overline{\partial_{\hat{A}}} \phi = 0$

$\Lambda F_{\hat{A}}=  \frac{i}{2} ({\lvert \phi  \rvert }^2 + s)$

These are the vortex equations on $\hat {L}$, but with the parameter $\tau$ replaced by minus
the scalar curvature. 
One can  perturb the above  equations by $-s+f$, when
$\beta = 0$, equations  reduce to the constant function vortex equations (see e.g.
\cite{G4}).

\section{{\bf Conclusion:}}

1. We have shown that for various vortex moduli spaces for compact  K$\ddot{\rm{a}}$hler  surfaces the quantum bundle is a holomorphic determinant line bundle.
If the compact K$\ddot{\rm{a}}$hler surface is projective, we have that the vortex moduli 
   spaces are projective, if they are smooth. There are obstructions to this result which we showed can be surmounted. Though the moduli space has been proven to be a Hilbert scheme in certain cases by Bradlow, we donot use the theory of Hilbert schemes.
   
   2. We also showed that on projective manifolds $(M, \omega)$, there are holomorphic determinant bundles (in the sense of Knusden-Mumford used by  Bismut, Gillet, Soul$\acute{\rm{e}}$) which play the role of 
the geometric quantum bundle, namely one for each  input data of a Hermitian  holomorphic line bundle $L$ of non-trivial Chern class on  a compact K$\ddot{\rm{a}}$hler manifold $Z$ (with Todd genus non-zero)  and a choice of a geometric quantization of $(M, \omega)$.

\section{ {\bf Acknowledgement:}}

We thank Professor Rukmini Dey for her patient guidance and help without which the manuscript would not have taken the present shape.

\end{document}